\documentclass[leqno]{amsart}
\usepackage{amssymb}
\newtheorem{Lem}{Lemma}
\newtheorem{Theo}{Theorem}
\begin{document}
\begin{center}
\title{The exponential sum over squarefree integers}
\author[J.-C. Schlage-Puchta]{Jan-Christoph Schlage-Puchta}
\end{center}
\maketitle
Denote by $r_\nu(N)$ the number of representations of $N$ as the sum
of $\nu$ squarefree numbers. In a series of papers Evelyn and Linfoot
\cite{EL1}--\cite{EL5} proved that
\[
r_\nu(N)=\mathfrak{S}_\nu(N)N^{\nu-1} +
\mathcal{O}(N^{\nu-1-\theta(\nu)+\varepsilon}),
\]
where
\[
\mathfrak{S}_\nu(N) = \frac{1}{(\nu-1)!}\left(\frac{6}{\pi^2}\right)^\nu
\prod_{p^2\nmid N}\left(1-\frac{1}{(1-p^2)^\nu}\right)
\prod_{p^2|N}\left(1-\frac{1}{(1-p^2)^{\nu-1}}\right),
\]
and
\[
\theta(2)=\theta(3)=\frac{1}{3},\qquad \theta(\nu)=\frac{1}{2}-\frac{1}{2\nu}
\quad(\nu\geq 4).
\]
Mirsky\cite{Mirsky} improved the error term for $\nu\geq 3$ to
$\theta(\nu)=\frac{1}{2}-\frac{1}{4\nu-2}$. Using a new approach to
bound the minor arc integral develloped by Br{\"u}dern, Granville,
Perelli, Vaughan and Wooley\cite{BGPVW}, Br{\"u}dern and
Perelli\cite{BP} showed that $\theta=\frac{1}{2}$ for all $\nu\geq 3$,
and that any further improvement would imply a quasiriemannian
hypothesis. Moreover, assuming the generalized riemannian hypothesis,
they proved that $\theta(3)=\frac{3}{4}+\frac{1}{14}$ and
$\theta(\nu)=\frac{3}{4}$ for all $\nu\geq 4$. These result are
optimal apart from the summand $\frac{1}{14}$; in personal
communication Br{\"u}dern conjectured
that $\theta(3)=\frac{3}{4}$ should hold true. It is the aim of this
note to prove this conjecture. 

Define $S(\alpha)=\sum_{n\leq N}\mu^2(n) e(\alpha n)$, and, for
integers $N$ and $Q$ satisfying $1\leq Q<N^{1/2}/2$, let $\mathfrak{M}(Q)$ be
the union of all intervals $\{\alpha: |\alpha q- a|\leq QN^{-1}\}$,
where $q\leq Q$, and $(a, q)=1$, and set $\mathfrak{m}(Q) = [QN^{-1},
1-QN^{-1}]\setminus\mathfrak{M}(Q)$. With these notation we will prove the
following. 
\begin{Theo}
\label{thm:minorarcs}
We have $S(\alpha)\ll N^{1+\varepsilon}Q^{-1}$ for all
$\alpha\in\mathfrak{m}(Q)$, provided that $Q\leq N^{1/2}$.
\end{Theo}
Under the restriction $Q\leq N^{3/7}$, this was proven in
\cite[Theorem~4]{BP}. As already remarked in \cite[Sec. 5]{BP}, the
weakening of the assumption on $Q$ implies the following.
\begin{Theo}
Assume the generalized riemannian hypothesis. Then we have
\[
r_3(N) = \mathfrak{S}(N)N^2 + \mathcal{O}(N^{5/4+\varepsilon}).
\]
\end{Theo}
By Dirichlet's theorem on diophantine approximation, for every
$\alpha\in\mathfrak{m}(Q)$ there exist coprime integers $a, q$ with
$q\leq NQ^{-1}$, such that $|q\alpha - a|\leq N^{-1}Q$. By the
definition of $\mathfrak{m}(Q)$, we necessarily have $q>Q$. Hence,
Theorem~\ref{thm:minorarcs} is essentially equivalent to the
following.
\begin{Theo}
\label{thm:main}
Define $S(\alpha)$ as above, and let $q$ be an integer satisfying
$|\alpha q- a|\leq q^{-1}$. Then we have 
\[
|S(\alpha)|\ll N^{1+\varepsilon} q^{-1} + N^\varepsilon q.
\]
\end{Theo}
We approach Theorem~\ref{thm:main} by the following lemma, which
replaces Lemma~1 in \cite{BP}.
\begin{Lem}
Let $\alpha\in(0, 1)$ be a real number, and assume that
$|q\alpha-a|<\frac{1}{q}$. Let $D$ be an integer, and denote by $W(D,
z)$ the number of integers 
$d\leq D$ satisfying $\|d^2\alpha\|\leq z$. Then, for
$D^2>\frac{1}{4}q$, we have  
\[
W(D, z)\ll D^2q^{-1} + D^{1+\varepsilon} z^{1/2}.
\]
\end{Lem}
\begin{proof}
Cut the interval $[1, D^2]$ into $K=[D^2q^{-1}]+1$ intervals of length
$q$, where the last interval may be shorter. For $k\leq K$, let $a_k$ be the number
of integers $d$, such that $\|d^2\alpha\|\leq z$ and $kq\leq d^2<
(k+1)q$. Then $\sum_{k\leq K} a_k = W(D, z)$, and by the
arithmetic-quadratic mean inequality, $\sum_{k\leq K} a_k^2\geq W(D,z)^2
K^{-1}$. Denote by $\mathcal{D}$ the set of all pairs $(d_1, d_2)$
with the properties that $\|d_i^2\alpha\|\leq z$ and
$1\leq|d_1^2-d_2^2|\leq q$. Then either $W(D, z)\leq 2K$, which is
sufficiently small, or we can bound $|\mathcal{D}|$ from below via
\[
|\mathcal{D}|\geq \sum_k \binom{a_k}{2}\gg \sum_k a_k^2 - \sum_k a_k
\gg \sum_k a_k^2 \gg W(D, z)^2 K^{-1}.
\]
Denote by
$\mathcal{N}\subseteq[1, q]$ the set of all values of $|d_1^2-d_2^2|$,
where $d_1, d_2$ ranges over all pairs in $\mathcal{D}$. Then every pair in
$\mathcal{D}$ gives rise to an element of $\mathcal{N}$, and the
number of different pairs $d_1, d_2$ having the same difference
$d_1^2-d_2^2=n$ is bounded above by the number of divisors of $n$,
and therefore $\ll q^\varepsilon$. Hence, we decuce
\[
W(D, z)^2 \ll |\mathcal{D}|K\ll |\mathcal{N}| K q^{\varepsilon}.
\]
On the other hand, for every $n\in\mathcal{N}$, we have
$\|n\alpha\|\leq \|d_1^2\alpha\|+\|d_2^2\alpha\|\leq 2z$, hence,
\[
W(D, z)^2 \ll D^2 q^{\varepsilon-1}\big|\big\{n\leq q: \|\alpha n\|\leq
2z\big\}\big|\ll D^2 q^{\varepsilon-1} (qz+1).
\]
From this we obtain in the case $W(D, z)> 2K$, that 
\[
W(D, z)\ll D^{1+\varepsilon} z^{1/2} + D^{1+\varepsilon}q^{-1/2},
\]
which is again of the right size, since $D>\frac{1}{2}q^{1/2}$.
\end{proof}
\begin{proof}[Proof of Theorem~$\ref{thm:main}$]
Write 
\begin{eqnarray*}
S(\alpha) & = & \sum_{d\leq\sqrt{N}}\mu(d)\sum_{m\leq Nd^{-2}}
e(\alpha d^2 m)\\
 & \ll & \log N \max\limits_{1\leq D\leq \sqrt{N}/2} \sum_{D\leq d<2D}
\min\Big(\frac{N}{D^2}, \|\alpha d^2\|^{-1}\Big)\\
 & = & \log N \max\limits_{1\leq D\leq \sqrt{N}/2} \Upsilon(\alpha, D),
\end{eqnarray*}
say. To prove Theorem~\ref{thm:main}, it suffices to show that
$\Upsilon(\alpha, D)\ll N^{1+\varepsilon} Q^{-1}$ for all $D\leq
\sqrt{N}/2$. For $D>\frac{1}{4}q^{1/2}$, we have 
\begin{eqnarray*}
\Upsilon(\alpha, D) & \ll & \log N\max\limits_{z>N/D^2} z^{-1} W(D, z)\\
 & \ll & \log N\max\limits_{z>N/D^2}\Big(z^{-1} D^2q^{-1} +
D^{1+\varepsilon} z^{-1/2}\Big)\\
 & \ll & N^{1+\varepsilon} q^{-1} + N^{1/2+\varepsilon}.
\end{eqnarray*}
For $D\leq \frac{1}{4}q^{1/2}$, we argue as in the proof of
\cite[Lemma~1]{BP}. We have
\[
|\alpha d^2 - ad^2/q|\leq 4D^2|\alpha-a/q| \leq 4D^2q^{-2} \leq \frac{1}{4q},
\]
and therefore
\[
|\Upsilon(\alpha, D)| \leq 2 \sum_{D\leq d< 2D}
\left\|\frac{ad^2}{q}\right\| \ll q\log q \ll N^\varepsilon q.
\]
Taking these estimates together, we find that
\[
S(\alpha) \ll N^{1+\varepsilon}q^{-1} + N^{1/2+\varepsilon} + N^\varepsilon q,
\]
and the second term is always dominated by either the first or the last one,
which implies our theorem.
\end{proof}

Jan-Christoph Schlage-Puchta\\
Mathematisches Institut\\
Eckerstr. 1\\
79104 Freiburg\\
Germany\\
jcp@mathematik.uni-freiburg.de
\end{document}